\documentclass[preprint,12pt]{elsarticle}
\usepackage{epsfig,graphicx,times,amsmath,amsthm, amssymb, multicol}
\usepackage[T1]{fontenc}  % Use fonts coded with 8 bits, not 7
\usepackage{calc}
\usepackage{amsmath}
\usepackage{amsfonts}
\usepackage{amssymb}
\usepackage{multirow}
\usepackage{epstopdf}
\newcommand{\R}{\mathbb{R}}

\newcommand{\mb}{\mathbf }
\newcommand{\mc}{\mathcal}

\newtheorem{theorem}{Theorem}
\newtheorem{lemma}[theorem]{Lemma}
\newtheorem{Ver}[theorem]{Verification}

\newtheorem{definition}[theorem]{Definition}
\begin{document}
\begin{frontmatter}

\title{ Reduced variable  optimization methods via implicit functional dependence with applications}
\author[um1,qub1]{Christopher G. Jesudason\corref{cor1}\fnref{fn1}} 
\ead{jesu@um.edu.my,chrysostomg@gmail.com}
\address[um1]{Department of  Chemistry and Center for Theoretical and Computational Physics,
Faculty of Science, University of Malaya,
50603 Kuala Lumpur, Malaysia}
 
\address[qub1]{Atomistic Simulation Centre (ASC),
School of Mathematics and Physics,
Queen's University Belfast,
Belfast BT7 1NN,
United Kingdom}

\fntext[fn1]{Sabbatical research at QUB, 2012-2013;Tel:+447587501616.}

% \footnotetext{Abbreviated title:\,\,  {\normalsize Induced parameter-dependent optimization method} }

% \date{\today}
%%\maketitle
%%Form of submission: Article \\
\begin{abstract}  Optimization methods have been broadly applied  to two classes  of objects \textit{viz.} (i) modeling and  description of data and (ii) the determination of the stationary points of functions. The overwhelming majority of these methods treat explicitly  the variables  relevant to the function optimization, where these parameters  are independently varied with a weightage factor. Here, a theoretical basis  is developed  which suggests algorithms that optimizes an arbitrary  number of variables for classes (i) and (ii) by the minimization of a function of a single variable deemed the most significant on physical or experimental grounds form which all other variables can be derived.    Often, one or more key variables play a physically more significant role than others even if all the variables are equally weighted  so that these key optimized variables serve as criteria  whether or not  reasonable solutions have been obtained, independently of the other variables.   Algorithms that focus on a reduced variable set  also avoid problems associated with multiple minima and 
maxima  that arise because of the large   numbers of parameters, and could increase  the accuracy 
of  the determination  by cutting down on  machine errors. The methods described  could have potentially significant  applications in the physical sciences where the optimization of one physically significant variable has priority over the other  variables as in the energy-position optimization schemes of computational quantum mechanics applied to structure determination, and in  chemical kinetics trajectory  calculations where the trajectory variable  is linked to many other variables of secondary importance.  For (i), we develop both an approximate but computationally more tractable  method  \textit{and} an exact method where the single controlling variable $k$ of all the other variables $\mb{P}(k)$ passes through the local stationary point of the least squares (LS) metric. For (ii), an exact theory is developed whereby the optimized function of an independent variation of all parameters coincides with that due to single parameter optimization. The implicit function theorem has to be further qualified to arrive at this result.  The topology of the surfaces of constant value of the target or cost function are considered for all the methods. A real world application of the above implicit methodology  to  rate  constant and final concentration parameter determination  for first  and second order chemical reactions from published  data  is attempted to illustrate its utility. This work is different from and more general than all the    reduction schemes for  conditional linear parameters used for example in extracting data from mixed signal spectra of  physical quantities such as found in laser spectroscopy  since it is valid for  conditional and nonconditional nonlinear parameters. Nor is it a subset of the Adomian decomposition method (ADM) used for estimating solutions of differential equations, which still require boundary conditions that do not feature in topics (i) and (ii).       
\end{abstract}
\begin{keyword}
Parameter fitting \sep  Chemical rate determination  \sep Single variable optimization \sep Generalization of conditional linear parameter optimization
%%
%\textbf{Keywords}:
%Parameter fitting; \,  Chemical rate determination;  \, Single variable optimization; \, Generalization of conditional linear parameter optimization
\MSC 41-04 \sep 41A28  \sep 41A58  \sep 41A63 \sep 41A65
\end{keyword}

\end{frontmatter}
\textbf{MSC code} 41-04 \, 41A28  \, 41A58  \, 41A63 \, 41A65

\textbf{PACS code} 05.10.-a \,07.05.Bx \, 02.60.Pn \, 02.60.Cb \, 82.20.Nk
%\end{keyword}
%%\end{frontmatter}

\section{Introduction} \label{sec:1}
 The following theory and elaboration revolve about properties of constrained and unconstrained  functions that are continuous and differentiable to various specified degrees\cite{crav1,depree1}, and the existence of implicit functions \cite{apos1} and the form of the function to be optimized. The implicit function theorem  is applied in a way that requires further qualified  because the optimization problem is  of an unconstrained kind  without any redundant variables. Methods (i)a,b (Secs.(\ref{s:m(i)a} ) and (\ref{s:m(i)b}) respectively) refers to modeling of data \cite[Chap.15,p.773-806]{nrc} where the form of the function $Q_{MD}(\mb{P},k)$ with independently varying variables $(\mb{P},k)$ have the general form
 \begin{equation} \label{e:nn1}
Q_{MD}(\mb{P},k)=\sum_{i=1}^{N_c} (y_i-f(\mb{P},t_i,k ))^2
\end{equation}
where $y_i$ and $t_i$  are datapoints and $f$ a known function, and optimizations of $Q_{MD}$ may be termed a least squares (LS) fit over parameters $(\mb{P},k)$ which  are independently optimized. Method (ii) focuses on optimizing  a general $Q_{Opt}(\mb{P},k)$ function, not necessarily LS in form. There are many standard and hybrid methods to deal with such optimization \cite[Ch.10]{nrc}, such as golden section searches in 1-D, simplex methods over multidimensions \cite[p.499-525]{nrc}, steepest descent and conjugate methods \cite{jan1} and variable metric methods in multidimensions \cite[p.521-525]{nrc} . Hybrid methods include multidimensional (DFP) secant methods \cite{dvd1}, BFGS (secant optimization) \cite{broy1}  and RFO (rational function optimization) \cite{ban1}  which is a Newton-Raphson technique utilizing a rational function rather than a quadratic model for the function close to the solution point. Global deterministic optimization schemes combine several  of the above approaches \cite[sec 6.7.6]{wales1}  Other ad hoc, physical methods  perhaps less easy to justify analytically include probabilistic "basin-hopping" algorithms \cite[sec 6.7.4]{wales1}, simulated annealing methods \cite{kirk1}  and genetic algorithms \cite[p.346]{wales1}. An analytical justification  on the other hand is attempted here, but in real-world problems some of the assumptions (e.g. $\mc{C}^2$ continuity, compactness of spaces) may not always obtain. For what follows, the distance metric used are all Euclidean, represented by $| \cdot |$ or $\| \cdot \|$ where $det\,\|\cdot \|$ represents the determinant of the matrix $|\cdot \|$. 
Reduction  of the number of variables  to be optimized is possible in the standard matrix regression model only if conditional linear parameters $\boldsymbol{\beta}$ exists \cite{batwat1}, where these $\boldsymbol{\beta}$  variables do not appear in the final $S(\boldsymbol{\theta})$ expression of the least squares function (\ref{e:x2}) to be optimized, whereas  the $\boldsymbol{\phi}$ nonconditional linear parameters do and are a subset of the $\boldsymbol{\theta}$ variables; for the existence of each conditional linear parameter, there is a unit reduction in the number of independent parameters to be optimized. These reductions  in variable number occurs for any  "expectation function" $f(\mathbf{x}, \boldsymbol{\theta})$ which is the model  or law for which a fitting is required, where there are $N$ different  datapoints  $\mathbf{x_i},\,i=1,2,\ldots N$ that must be used to determine the $p$ parameter variables $\boldsymbol{\theta}$ \cite[p.32,Ch.2]{batwat1}.  A conditionally  linear parameter $\theta_i$ exists if and only if the derivative of the expectation function $f(\mathbf{x}, \boldsymbol{\theta})$ with respect to $\theta_i$  does not  involve - in other words is independent of - $\theta_i$. Clearly such a condition may severely limit the  number of parameters that can be neglected for  the expectation function variables when the prescribed  matrix regressional techniques are employed \cite[Sec.3.5.5, p.85]{batwat1} where the residual sum of squares is minimized:
\begin{equation}\label{e:x2}
S(\boldsymbol{\theta})=\left\|\mathbf{y}-\boldsymbol{\eta}(\boldsymbol{\theta})\right\|^2.
\end{equation}
 The $N$-vectors $\boldsymbol{\eta(\theta)}$ in $P$ dimensional space defines the expectation surface. If the $\boldsymbol{\theta}$  variables  are  partitioned    into the conditional  linear parameters $\boldsymbol{\beta}$ and the other nonlinear parameters $\boldsymbol{\phi}$, then the response can be written $\boldsymbol{\eta(\beta,\phi)=A(\phi)\beta}$. Golub and Pereyra \cite{Gol1} used standard Gauss-Newton algorithm to minimise $S_2(\boldsymbol{\phi})=\left\|\mathbf{y}-\boldsymbol{A(\phi)}\boldsymbol{\hat{\beta}(\phi)}\right\|^2$  that depended only on the nonlinear parameters $\boldsymbol \phi $ where $\boldsymbol{\hat{\beta}(\phi)}=\boldsymbol{A^+(\phi)}\boldsymbol{y}$ with $\boldsymbol{A^+}$ being a defined pseudoinverse of $\boldsymbol{A}$ \cite[Sec.3.5.5,p.85]{batwat1} where $\boldsymbol{A^+}$ and $\boldsymbol{A}$ are matrices. The variables must be separable as discussed above and the number of variable reduction is only equal to the number of conditional linear parameters that exists for the problem. In applications, the preferred algorithm that exploits this valuable variable reduction is called Variable Projection.   There are many applications in time resolved spectroscopy  that is heavily dependent on this technique and   many references to the method are given in the review by  van Stokkum et al. \cite{stokkum1}.  Recently this method of variable projection has been extended  in a restricted sense  \cite{shear1} in the field of inverse problems, which is not related to our method of either modeling or optimization, nor is the methodology related to the implicit function methods. In short, much of the reported methods developed are $ad \,hoc$, meaning that they are constructed to face the specific problems at hand with no pretense to any all-encompassing generality and this work too  is $ad \,hoc$  in the sense of suggesting variable reduction with   specific  classes of non-inverse problems  as  indicated where the work develops a method  of reducing the variable number to unity for  all  variables in the expectation function space  \emph{irrespective of whether they are  conditional or not}    by approximating their values by a method of averages (for method(i)a) without any form of linear regression being used in determining their approximations during the minimization cycles, and without necessarily using the standard  matrix theory that is valid for a very limited class of functions.  Methods (i)b and (ii) are exact treatments. No  "`eliminating"' of  conditional linear parameters are involved  in this nonlinear regression method because they are explicitly calculated. Nor is any  projection in the mathematical sense involved.   These more general methods could  have useful applications  in deterministic systems comprising many parameters that are all linked to one variable, the primary one (denoted $k$ here) that is considered on physical grounds to be the most important one. A generalization of this method would be to select a smaller set of variables than the full parameter list.   Examples of multiparameter complex systems include those for multiple-step elementary reactions each with its own rate constant that gives rise to photochemical spectra signals that must be resolved unambiguously \cite{getoff1}. All these complex and coupled processes in physical theories are related by postulated laws $Y_{law}(\mathbf{P},k,t)$ that feature parameters $(\mathbf{P},k)$. Other examples include quantum chemical calculations  with  many topological and orientation variables  that need to be optimized with respect to the energy, but in relation to one or  a few variables, such as the molecular trajectory parameter during  a chemical reaction where this variable is of primary significance in deciding on the 'reasonableness' of the analysis \cite[Sec. 6.2.3,p.294]{wales1}. Method (i)a  and (i)b below    refer to LS data-fitting algorithms. Method (i)a is an approximate method where it is proved under certain conditions that it could be a more accurate determination of parameters compared to a standard LS fit using (\ref{e:nn1}).   Method (i)b develops the methodology whereby its optimum value for $Q_{MD}$  with domain values $(\mb{P},k)$ coincides with that of the standard LS method where the $(\mb{P},k)$ variables are varied independently. Also discussed are the relative accuracy of both methods (i)a in  subsection (\ref{subsec:1a})and (i)b  (endnote at end of section \ref{s:m(i)b}). Method (ii) develops a  single parameter optimization where the  conditions of an arbitrary  $Q_{OPT}(\mb{P},k)$ function are met simultaneously, \textit{viz.}
\[ \frac{\partial Q_{OPT}(k)}{\partial k} =0 \rightarrow  \left \{       \frac{\partial Q_{OPT}(\mb{P},k)}{\partial \mb{P}}=0, \frac{\partial Q_{OPT}(\mb{P},k)}{\partial k}=0   \right \}.\]
We note that methods (i)a, (i)b and (ii) are not related  to the  \mbox{Adomian} decomposition method and its variants that expands polynomial coefficients \cite{waz1} for solutions to differential equations not connected to estimation theory; indeed here there are no boundary values that determine the solution of the differential equations.
\section{Method (i)a theory} \label{s:m(i)a}
Deterministic laws of nature are sometimes written - for the simplest examples- in the form 
\begin{equation} \label{eq:n1}
Y_{law}= Y_{law}(\mathbf{P},k,t)
\end{equation}
linking the variable $Y_{law}$ to $t$. The components of $\mathbf{P}$, $P_i (i=1,2,...N_p)$ and $k$ are parameters. Verification of a law of form (\ref{eq:n1})  relies on an experimental dataset $\{ (Y_{exp}(t_i), \, t_i), i=1,2,...N) \}$. The $t$ variable could be a vector of variable components  of experimentally measured values or a single parameter as in the example below where $t_i$ denotes values of time $t$ in the domain space. The vector form will be denoted $\mathbf{x}$. Variables $(\mathbf{x})$  are defined as members of the 'domain space' of the measurable system  and  similarly  $Y_{law}$ is  the defined  range or 'response' space of the physical measurement.    Confirmation or verification of the law is based on (a) deriving  experimentally  meaningful  values for the parameters $(\mathbf{P},k)$ and (b) showing a good enough degree of fit between  the experimental set $Y_{exp}(t_i)$  and $Y_{law}(t_i)$. In real world applications, to chemical kinetics for instance, several methods  \cite[etc.]{hou1,moore1,went1,went2} have been devised to determine the optimal $\mathbf{P},k$ parameters, but most if not all these methods consider the aforementioned parameters as autonomous and independent (e.g. \cite{moore1}). A similar scenario broadly holds  for current state of the art applications of structural elucidation via energy functions \cite{wales1}.  To preserve the viewpoint of the inter-relationship  between these parameters and the experimental data, we devise a scheme  that relates $\mathbf{P}$ to $k$ for all $P_i$ via  the set $\{Y_{exp}(t_i),t_i\}$, and optimize the fit over $k$-space only. i.e. there is induced a $P_i(k)$ dependency on $k$ via the the experimental set $\{ Y_{exp}(t_i),t_i\}$. The conditions that allow for this will also  be stated.  
\subsection{Details of method (i)a} \label{subs:2}
Let  $N^\prime$ be  the number of dataset pairs $\{Y_{exp}(t_i),t_i\}$, $N_p$ the number of components of the $\mathbf{P}$ parameter, and  $N_s$ the number of singularities where the use of a particular dataset $(Y_{exp}, t)$ leads to a singularity in the determination of $\bar{P}_i(k)$ as defined below and which must be excluded from being used in the determination of $\bar{P}_i(k)$. Then $(N_p+1)\leq (N-N_s)$ for the unique determination of $\{\mathbf{P},k\}$. Let $N_c$ be the total number of different datasets that can be chosen which does not lead to singularities. If  the singularities are not choice dependent, i.e. a particular data-set pair leads to singularities  for all possible choices, then we have the following definition for $N_c$ where  $^{N-N_s}C_{N_p}=N_c$ is the total number of combinations of the data-sets $\{Y_{exp}(t_i),t_i\}$ taken $N_p$ at a time that does not lead to singularities in $P_i$.  In general, $N_c$ is determined by the nature of the data sets and the way in which  the proposed equations are to be solved. Write $Y_{law}$ in the form 
\begin{equation} \label{eq:n2}
Y_{law}(t,k)=f(\mathbf{P}, t,k) 
\end{equation}
and for a particular dataset $\{ Y_{exp}(t_i),t_i\}$, write $f(i)\equiv f(\mathbf{P}, t_i,k)$. Define the  vector function $\mathbf{f_g}$ with components $f_g(i)\equiv Y_{exp}(t_i)-f(i)=f_g(i)(\mathbf{P},k)$. Assume $\mathbf{f_g}\in \mathcal{C}^1$ defined on an open set $K_0$ that contains $k_0$. 
\begin{lemma}\label{l1}
For any such $k_0$ if $det \left \| \frac{\partial f_g(i)(\mathbf{P},k_0) }{\partial P_j} \right \|\neq 0$, $\exists$ the unique function $\mathbf{P}(k)\in \mathcal{C}^1$ (with components $P_i(k)\ldots P_{N_c}(k)$) defined on $K_0$ where $\mathbf{P}(k_0)=\mathbf{P_0}$, and where $\mathbf{f_g}(\mathbf{P}(k),k)=0$ for every $k\in K_0$. 
\end{lemma} 

\begin{proof}
The above follows from the Implicit function theorem (IFT) \cite[Th.13.7,p.374]{apos1} where $k\in K_0$ is the independent variable  for the existence of the $\mathbf{P}(k)$ function.
\end{proof}
We seek the solutions for $\mathbf{P}(k)$ subject to the above conditions for our defined functions. 
Map $f\longrightarrow Y_{th}(\bar{\mathbf{P}}, t,k)$ as follows
\begin{equation} \label{eq:n3}
Y_{th}(t,k)= f(\bar{\mathbf{P}}, t,k)
\end{equation}
where the term $\bar{\mathbf{P}} $  and its components are   defined below and where $k$ is a varying parameter.
For any of the $(i_1, i_2, \ldots, i_{N_p} )$ combinations denoted by a combination variable $\alpha \equiv (i_1, i_2, \ldots, i_{N_p} ) $ where $i_j \equiv (Y_{exp}(t_{i_j}) , t_{i_j})$ is a particular dataset pair, it is in principle possible to solve for the components of  $\bar{\mathbf{P}}$ in terms of $k$ through the following simultaneous equations: 

\begin{equation} \label{eq:n4}
\begin{array}{rll}
Y_{exp}(t_{i_1})&=& f(\mathbf{P}, t_{i_1} , k) \\[0.5cm]
Y_{exp}(t_{i_2})&=& f(\mathbf{P}, t_{i_2} , k) \\
 
  &\vdots&  \\
  Y_{exp}(t_{i_{N_p}})&=& f(\mathbf{P}, t_{i_{N_p}} , k)
 	\end{array}
\end{equation}
 from Lemma \ref{l1}. And each $\alpha$ choice yields a unique solution  $P_i(k,\alpha)\,(i=1,2 \ldots N_p) $ where $P_i(k,\alpha)\in \mathcal{C}^1$. Hence any functions of $P_i(k,\alpha)$ involving addition and multiplication are also in $\mathcal{C}^1$.  For each $P_i$, there will be $N_c$ different solutions, $P_i(k,1),P_i(k,2), \ldots P_i(k,N_c)$ . We can define an arithmetic mean (there are several possible mean definitions that can be utilized)  for the components of $\bar{\mathbf{P}}$ as 
\begin{equation} \label{eq:n5b}
\bar{P_i}(k)= \frac{1}{N_c}\sum_{j=1}^{N_c}P_i(k,j) .
\end{equation}
In choosing an appropriate functional form for $\bar{\mathbf{P}}$ (eq.\ref{eq:n5b}) we assumed equal weightage for each of the dataset combinations; however the choice is open, based on appropriate physical criteria. We verify  below  that the choice of $\bar{\mathbf{P}}(k)$ satisfy the  constrained variation of the  LS method so as to emphasize the connection between the level-surfaces of the unconstrained LS with the line function $\bar{\mathbf{P}}(k)$ .  
\begin{figure}[htbp]
\includegraphics[width=6cm]{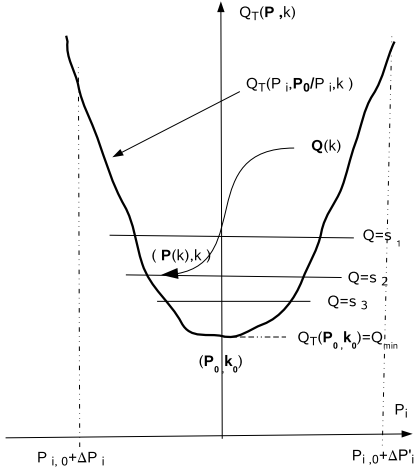}
\caption{ Depiction of how the $k$ variation optimizing $Q$ leads to a solution on a level surface of the  $Q_T$ function where $Q_T \le Q$.}
\label{fig:4n}
\end{figure}
%bk1

Each $P_i(k,j)$ is a function of $k$ whose derivative is known either analytically or by numerical differentiation.
To derive an optimized set, then for the LS method, define 
\begin{equation} \label{eq:n5}
Q(k)= \sum_{i=1^\prime}^{N^\prime} (Y_{exp}(t_i)-Y_{th}(k,t_i) )^2.
\end{equation}

Then for an optimized $k$, we have $Q^\prime(k)=0.$ Defining 
\begin{equation} \label{eq:n6}
R(k)=\sum _{i=1^\prime}^{N^\prime}(Y_{exp}(t_i)- Y_{th}(k,t_i)  ).Y_{th}^\prime(k,t_i)
\end{equation}
the optimized solution of $k$ corresponds to $R(k)=0$  which  has been reduced to a  one dimensional problem. The standard LS variation  on the other hand  states that the variables $P_T=\{\mathbf{P},k\}$ in (\ref{eq:n2}) are independently varied so that 
\begin{equation} \label{e:6n}
Q_T(P_T)=\sum_{i=1}^{N^\prime} (Y_{exp}(t_i)-f(P_T,,t_i ))^2.
\end{equation}
with solutions for $Q_T$  in terms of $P_T$ whenever $\partial Q_T/\partial P_T=0$. Of interest is the relationship between the single variable variation  in (\ref{eq:n5} ) and the total variation in (\ref{e:6n}). Since $\overline{\mathbf{P}}$ is a function of $k$, then (\ref{e:6n}) is a constrained variation      where
\begin{equation} \label{e:8n}
\delta Q(k)=\delta Q(\overline{\mathbf{P}},k)=\left( \frac{\partial Q}{\partial \overline{\mathbf{P}}}.\delta\overline{\mathbf{P}} +\frac{\partial Q}{\partial k} \right )
\end{equation}
subjected to $g_i(\overline{\mathbf{P}},k)=\bar{P}_i-h_i(k)=0$ (i.e. $ \bar{P}_i=h_i(k)$ for some function of k) and where $\bar{P}_i$ are the components  of $\bar{\mathbf{P}}$. According to the Lagrange multiplier theory \cite[Th.13.12,p.381]{apos1} the function $f:\R^n\rightarrow \R$ 
  has an optimal value at $x_0$ subject to the constraints $g:\R^n\rightarrow \R^m$ over the subset $S$ where
 $\mathbf{g}=(g_1,g_2\ldots g_m)$ vanishes, i.e.  $x_0 \in  X_0 $ where $X_0=\{ x:x \in S, \mathbf{g}(x)=0\}$ when  either of the following  equivalent equations (\ref{e:12n},\ref{e:13n}) are satisfied
\begin{eqnarray}
 D_rf(x_o)&+&\sum_{k=1}^m \lambda _k D_rg_k(x_0)=0 \,\,(r=1,2,\ldots n)   \label{e:12n}\\
  \nabla f(x_0)&+&\lambda_1\nabla g_1(x_0)+\ldots    \lambda_m\nabla g_m(x_0) = 0   \label{e:13n}  
\end{eqnarray}
where $det \| D_jg_i(x_0)\|\neq 0$ and the $\lambda$'s are invariant real numbers. We refer to $\bar{P}_i$  as any variable that is a function of $k$ constructed on physical or mathematical grounds, and not just to the special defined in (\ref{eq:n5b}). Write 
\begin{equation} \label{e:14n}
g_i=\bar{P}_i-p_i(k)=0\,(i=1,2,\ldots N_p)
\end{equation}
where $D_jg_i(x_o)=\delta_{ij}$ since  $D_j=\partial/\partial \bar{P}_j$ and therefore $det \|D_jg_i\neq0\|$.
We abbreviate the functions $f(i)=f(\mathbf{P},t_i,k)$ and  $\bar{f}(i)=f(\mathbf{\bar{P}},t_i,k)$. Define  
\begin{equation} \label{e:15n}
f_Q(x)\equiv Q(\bar{\mathbf{P}},k,t)= \sum_{i=1'}^{N'}(Y_{exp}(t_i)-\bar{f}(i))^2
\end{equation}
where $Y_{exp}(t_i)$ are the experimental subspace variables as in (\ref{eq:n4}) with $x\in X_0$ defined above. We next  verify the relation between $Q(k)$ and $Q_T$ . 
\begin{Ver}
The solution $Q^\prime(k)$=R(k)=0 of (\ref{eq:n6} ) is equivalent  to the variation of $f_Q(x)$ defined in (\ref{e:15n}) subjected to constraints $g_i$ of (\ref{e:14n}).
\end{Ver}

\begin{proof}
Define the Lagrangian to the problem as $\mathcal{L}=f_Q(x)+\sum_{i=1}^{N_p}\lambda_i g_i$. Then the equations that satisfy the stationary condition 
\begin{equation} \label{e:16n}
\frac{\partial \mathcal{L}}{\partial \bar{P}_j}=0 \,\,(j=1,2,\ldots N_p);\frac{\partial \mathcal{L}}{\partial k}=0
\end{equation}
reduces to the  (equivalent) simultaneous equations 
\begin{eqnarray}
      \sum_{i=1}^{N'} (Y_{exp}(t_i)&-&\bar{f}(i))\frac{\partial \bar{f}(i)}{\partial \bar{P}_j} =\lambda^\prime_j \,\,(j=1,2,\ldots N_p)          \label{e:17n}\\
\sum_{i=1}^{N'} (Y_{exp}(t_i)&-&\bar{f}(i))\frac{\partial \bar{f}(i)}{\partial k} +\sum_{j=1}^{N_p}\lambda^\prime_j\frac{\partial p_j}{\partial k}= 0   \label{e:18n}  
\end{eqnarray}
Substituting $\lambda ^\prime_j$ in (\ref{e:17n}) to (\ref{e:18n})  leads to 
\begin{equation} \label{e:19n}
\sum_{i=1}^{N'}(Y_{exp}(t_i)-\bar{f}(i))\frac{\partial \bar{f}(i)}{\partial k}+\sum_{j=1}^{N_p}\sum_{i=1}^{N'}(Y_{exp}(t_i)-\bar{f}(i)).\frac{\partial \bar{f}(i)}{\partial \bar{P}_j}.\frac{\partial p_j}{\partial k}=0.
 \end{equation}
Since  $\frac{d\bar{P}_i}{dk} = \frac{\partial\bar{p}_i}{\partial k}$, then  (\ref{e:19n}) is equal to  $\frac{dQ(\bar{\mathbf{P}},k,t)}{dk}=0 $ of   the $Q$ functions in (\ref{e:6n},\ref{e:8n} and \ref{e:15n}).
\end{proof}
Of interest is the theoretical relationship of the $\bar{\mathbf{P}},k$ variables of the $Q$ function  described by (\ref{e:8n},\ref{eq:n5}) denoted $Q_1$ and those of the free variational $Q$ function of (\ref{e:15n}) denoted $Q_2$ with the variable set which can be written
\begin{eqnarray}
              Q_1 &=& Q( \bar{\mathbf{P}},t,k) \label{e:20n }\\
             Q_2 &=& Q( \mathbf{P},t,k)   \label{e:21n  }  
\end{eqnarray}
which is given by  the following theorem, where we abbreviate $\alpha_i=(Y_{exp}(t_i) -f(\mathbf{P},t_i,k) )$ and $\bar{\alpha}_i=(Y_{exp}(t_i) -f(\bar{\mathbf{P}},t_i,k) )$ where we note that the $f$ functional form is unique and of the same form  for both these $\alpha$ variables.
\begin{theorem}\label{t:3}
The unconstrained LS solution to $Q_2=Q(\mathbf{P},k,t)$ for the independent variables $\{\mathbf{P},k \}$ is also a solution for  the constrained variation single variable $k'$ where $\mathbf{P}=\bar{\mathbf{P}}(k'),k=k'$. Further, the two solutions coincide if and only if 
\[  \sum_{i=1}^{N'}  (Y_{exp}(t_i) -f(\bar{\mathbf{P}},t_i,k) )\frac{\partial f(\bar{\mathbf{P}},t_i,k) }{\partial \bar{P}_j}  \,, j=1,2,\ldots N_p.\] 
\end{theorem}

\begin{proof}
The $Q_2$ unconstrained solution is derived from the equations 
\begin{eqnarray}
   \frac{\partial Q_2}{\partial P_j}&=& c.\sum_{i=1}^{N'}\alpha_i\frac{\partial f(i)}{\partial P_j}= 0    \, ,j=1,2,\ldots N_p             \label{e:22n}\\
 \frac{\partial Q_2}{\partial k}&=& c.\sum_{i=1}^{N'}\alpha_i\frac{\partial f(i)}{\partial k}= 0              \label{e:23n}  
\end{eqnarray}
with $c$ being constants. If there is a $\bar{\mathbf{P}}(k)$ dependency, then we have 
\begin{equation} \label{e:24n}
\frac{dQ_1(k)}{dk}=c.\sum_{j=1}^{N_p}\left(\sum_{i=1}^{N'}\bar{\alpha}_i\frac{\partial \bar{f}(i)}{\partial P_j}  \right)\frac{\partial \bar{P}_j}{\partial  k} 
+c.\sum_{i=1}^{N'}\bar{\alpha}_i\frac{\partial \bar{f}(i)}{\partial k}.
\end{equation}
If the variable set $\{ \mathbf{P},k\}$ satisfies (\ref{e:22n}) and (\ref{e:23n}) in unconstrained variation, then the values when substituted into (\ref{e:24n} ) satisfies the equation $\frac{dQ_1(k)}{dk}=0$ since $f(i)$ and $\bar{f}(i)$ are the same functional form. This proves the first part of the theorem. The second part follows from the converse argument, where from (\ref{e:24n} ), if  $\frac{dQ_1(k)}{dk}=0$, then setting one factor  to zero in (\ref{e:25n} ) leads to the implication of (\ref{e:26n})
\begin{eqnarray}
    \sum_{i=1}^{N'}\bar{\alpha}_i\frac{\partial \bar{f}(i)}{\partial\bar{P}_j} &=& 0\, ,j=1,2,\ldots N_p   \label{e:25n}\\
               &\Rightarrow & \nonumber\\ 
\sum_{i=1}^{N'}\bar{\alpha}_i\frac{\partial \bar{f}(i)}{\partial k}       &=&0    \label{e:26n}  
\end{eqnarray}
 the solution set $\{\bar{\mathbf{P}}(k'), k'\}$  which satisfies $\frac{dQ_1(k)}{dk}=0$ is satisfied by the conditions of both (\ref{e:25n}) and (\ref{e:26n}). Then (\ref{e:25n}) satisfies (\ref{e:22n}) and  (\ref{e:26n}) satisfies (\ref{e:23n}).
\end{proof}

The theorem, verification and lemma above do not indicate topologically under what conditions a coincidence of solutions for the constrained and unconstrained models exists. Fig.(\ref{fig:4n}) depicts the discussion below. From theorem (\ref{t:3}), if set $A$ represents the solution $\{\mb{P},k\}$ for the  unconstrained LS method and set $B=\{\mb{\bar{P}},k'\}$ for the constrained method , then $B\supseteq A$. 
Define $k$   within the range $k_1\leq  k \leq  k_2$. Then $k$ is  in a compact space, and since $\mb{P}(k)\in \mc{C}^1$,$\mb{P}(k)$
is uniformly  continuous, \cite[TH.8,p.79]{depree1}. Then admissible solutions to the above constraint problem with the inequality $B\supseteq A$ implies $Q(\mb{P}(k)) \geq  Q_{min}$ where $Q_{min}$ is the unconstrained minimum. The unconstrained $Q=Q_T$ LS function to be minimized  in (\ref{e:6n}) implies 
\begin{equation} \label{e:A2}
\nabla Q=0\,\,(\frac{\partial Q}{\partial k}=\frac{\partial Q}{\partial P_i}=0 \, \mathrm{for}\,i=1,2,\ldots N_p )
\end{equation}
Defining the constrained function $Q_c(k)=\sum_{i=1}^{N^\prime}(Y_{exp}(i)-f(\mb{P}(k),t_i,k))^2$ then $Q_c(k)=Q\circ P_T$ where $P_T=(P_1(k),P_2(k),\ldots P_{N_p}(k),k)^T$. Because $Q^{\prime}_c(k)=\left( \frac{\partial Q}{\partial \mb{P}},\frac{\partial Q}{\partial k}  \right )\cdot \left ( P_1'(k)\ldots 
P_{N_p}(k),k\right)^T$, solutions occur when (i) $\nabla Q=0$ corresponding to the coincidence of the local minimum of the unconstrained $Q$  for the best choice for  the line with coordinates $(\mb{P}(k),k)$ as it passes  through the local unconstrained minimum and (ii) $P^\prime _i(k)=0,(i=1,2,\ldots N_p),\frac{\partial Q}{\partial k} =0$ where this solution is a special case of (iii) when the vector $ P^{\prime}_T$ is $\perp$  to $\nabla Q \neq 0$, i.e.  $P^{\prime}_T$  is at a tangent to the surface $Q=S_2$ for some $S_2  \geq Q_{min}$ where this situation is shown in Fig.(\ref{fig:4n}) where the vector is tangent at some point of the surface $Q_T=S_2$. Whilst the above characterizes the  topology of  a solution, the existence of a solution  for the line  $(\mb{P}(k),k)$  which passes through the point of the unconstrained minimum  of $Q_T$ is proven below  under certain conditions  where a set of equations are constructed  to allow for this exceptionally important application. Also discussed is the  case when it may be possible for unconstrained solution  set $U$  to  satisfy the inequality $Q_T(U)\geq Q_c$, where $Q_C$ is a function designed to accommodate all solutions of  (\ref{eq:n4}).

%%%%\sum_{i=1}^N'(Y_{exp}(t_i)-\bar{f}(i)).\frac{\partial \bar{f}(i)}{\partial %\bar{P}_j}\frac{\partial %%p_j}{\partial k}=0.
\subsection{Discussion of LS fit for a function $Q_c$  with  a possibility of a smaller LS deviation than  for $\{\mb{P},k\}$  parameters derived from a free variation of eq.(\ref{e:6n})}\label{subsec:1a}
The LS function metric such as (\ref{e:6n}) implied $Q(\mb{P},k)\leq Q(\mb{\bar{P}},k) )$ at a stationary (minimum) point for variables $(\mb{P}, k)$. On the  other hand, the sets of solutions (\ref{eq:n4}), $N_c$ in number provides for each set exact solutions $\mb{P}(k)$ averaged to $\mb{\bar{P}}(k)$. If the $\{\mb{\bar{P}}_i\},i=1,2,\ldots N_c$ solutions are in a $\delta$-neighbourhood, then it may be possible that the composite function metric to be optimized  over all the sets of equations $\{\}_i$, $N_c$ in number defined here 
\begin{equation} \label{e:N30}
Q_c(\mb{P},k)=\sum_{i\in \{\}_1, \{\}_2\ldots \{\}_{N_c} }(Y_{exp} -f(i)    )^2
\end{equation}
could be such that 
\begin{equation} \label{e:N31}
Q_c(\mb{P},k)\geq Q_c(\mb{\bar{P}},k)
\end{equation}
implying that, under these conditions, the $Q_c$ of (\ref{e:N30}) is a better measure of fit. For what follows, the $\mb{P}(k)_{\{\}_i}$ for equation set $\{\}_i$ obtains  for all $k$ values of the open set $\mc{S},\,\,k\in \mc{S}$, from the IFT, including $k_0$ which minimises (\ref{eq:n5}). Another possibility that will be discussed briefly  later  is where  in (\ref{e:N30}), all $\{\mb{P},k\}$ are free to vary. Here we consider the case of the $N_c$ $\mb{P}$ values averaged to $\mb{\bar{P}}$ for some $k$. We recall the Intermediate-value theorem (IVT) \cite[Th.4.38,p.87]{apos1}for real continuous functions $f$  defined over 	   a connected domain $S$  which is a subset of some  $\R^n$. We assume that the $f$ functions immediately below obey the IVT. For each $\{\mb{P}_i\}$  solution of the $\{\}_i$ set for a specific $k=k_0$ we assume that the function  $Q_{c,i}(\mb{P})$ is a strictly increasing function in the sense  of definition(\ref{d:1})below, where 
\begin{equation} \label{e:n31b}
Q_{c,i}(\mb{P})=\sum_{j\in \{\}_i} (Y_{exp}(j)-f(j))^2
\end{equation}
with $Q_{c,i}(\mb{P})(k_0)=0$, in the following sense:

\begin{definition}\label{d:1}
A real function f is (strictly) increasing on a connected domain $S\in \R^N$ about the origin at $r_0$ if relative to this origin, if $\left| r_2\right| >\left|r_1 \right|$ (of the boundaries $\partial$ of ball $B(r_0,r_1)$ and $B(r_0,r_2)$ implies both $\mathrm{max}\,f(\partial B(r_o,r_2)(>)\geq \mathrm{max}\,f(\partial B(r_o,r_1)$ and $\mathrm{min}\,f(\partial B(r_o,r_2)(>)\geq \mathrm{min}\,f(\partial B(r_o,r_1)$.
\end{definition}
\textit{Nb.}\,A similar definition obtains for a (strictly) decreasing function with the $(<)\leq $ inequalities. Since the $\partial B$ boundaries are compact, and $f$ continuous, the maximum and minimum values  are attained for all ball boundaries. 
\begin{lemma} \label{lm2}
For any region bounded by $\partial B(r_0,r_1$ and $\partial B(r_0,r_2$ with coordinate $\mb{r}$ (radius $r$ centered about coordinate $r_0$) ,
\begin{equation} \label{e:n32}
\mathrm{min} f(\partial B(r_0,r_1)< f(\mb{r}) < \mathrm{max} f(\partial B(r_0,r_2).
\end{equation}
\end{lemma}
\begin{proof} 
Suppose in fact $f(\mb{r})< \mathrm{min}f(\partial B(r_0,r_1)))$ then $\mathrm{min}f(\partial B(r_0,r)))< \mathrm{min}f(\partial B(r_0,r_1))) \,\,(r>r_1)$ which is a contradiction  and a similar proof obtains for the upper bound.
\end{proof}
\textit{Nb.}\, Similar conditions apply  for the non-strict inequalities $\leq ,\geq $ . The function that is optimized  is 
\begin{equation} \label{e:n32b}
Q_{c}(\mb{P})=\sum_{i=1}^{N_c}Q_{c,i}(\mb{P})
\end{equation}
Define $\mb{P_i}$ as the solution vector for  the equation set $\{\}_i$. We illustrate the conditions where the solution $\mb{P_T}$ for a free variation for the $Q$ metric  given in(\ref{e:6n}) can fulfill the inequality  where $Q_c$ is as defined in (\ref{e:n32b})
\begin{equation} \label{e:n33}
Q_c(\mb{P_T})\geq Q_c(\bar{\mb{P}})
\end{equation}
with $\bar{\mb{P}}$ given as in (\ref{eq:n5b}). A preliminary result is required. Define $\mathrm{max}\,\Vert \mb{P}_i-\mb{P}_j\Vert=\delta\,\,\forall \,i,j=1,2,\ldots N_c$ and $\delta \mb{P}_i=\bar{\mb{P}}-\mb{P}_i$.

\begin{lemma} \label{lm3}
$\Vert \delta \mb{P}_i \Vert \leq  \delta.$
\end{lemma}
\begin{proof} 
$\Vert \bar{\mb{P}}-\mb{P}_i \Vert=\frac{1}{N_c}\Vert \sum_q \mb{P}_q -N_C\mb{P}_i \Vert\leq \frac{1}{N_c}\sum_{j=1}^{N_c}\Vert \mb{P}_j-\mb{P}_i\Vert \leq \delta$.
\end{proof}
\begin{lemma} \label{lm4}
$Q_c(\mb{P}_T,k) > Q_c(\mb{\bar{P}},k)$ for $\delta <\Vert\mb{P}_T-\mb{P}_i\Vert < \delta_T$ for some $\delta_T$
\end{lemma}
\begin{proof} 
Any  point $\mb{\bar{P}}$ would be located within an spherical annulus centred at $\mb{P}_i$,  with radii so chosen so that by  lemma (\ref{lm2}), the following results:
\begin{equation} \label{e:n34}
\epsilon _{max,i} > Q_{c,i}(\mb{\bar{P}})> \epsilon _{min,i}
\end{equation}
where $f=Q_{c,i} $ in (\ref{e:n31b}). Choose $\delta_i$  so that $\delta_i<\Vert\delta\mb{P}_i<\delta$ . Define $Ann(\delta,\delta_i, \mb{P_i})$ as the space bounded by the  boundary of the balls centered on $\mb{P}_i$ of radius $\delta$ and $\delta_i$  ($\delta > \delta_i$ ). Then $\mb{\bar{P}}\in Ann(\delta,\delta_i, \mb{P_i})$ by lemma (\ref{lm3}). Since $Q_c$ in (\ref{e:N30}) is not equivalent to $Q_T=Q$ in (\ref{e:6n})where we write here the free variation vector solution as  $\mb{P}_T$, then the above results leads to the following:
\begin{eqnarray}
    \sum_{i=1}^{N_c} \epsilon _{min,i} &<& Q_c(\mb{\bar{P}}) = \sum_{i=1}^{N_c}Q_{c,i}(\mb{\bar{P}})<\sum_{i=1}^{N_c}\epsilon_{max,i} \label{e:n35}\\
            \epsilon _{max,i} &<&  Q_{c,i} (\mb{P_T}) < \epsilon _{T,i} \label{e:n36}  
\end{eqnarray} 
 where (\ref{e:n36}) follows from (\ref{e:n32}) . Summing  (\ref{e:n36}) leads to $Q_c(\mb{P}_T,k) > Q_c(\mb{\bar{P}},k)$.
\end{proof}
Hence we have demonstrated that it may be more realistic or accurate  to fit parameters based on  a function that represents different coupling sets  such as $Q_C$ above rather than the standard LS method using (\ref{e:N30}) if $\mb{P}_T$  lies  sufficiently  far away from $\mb{\bar{P}}$. We note that if $\mb{P_T}$ is the solution of the free variation of  the above $Q_c$ in (\ref{e:N30}), then from the arguments presented after the proof of theorem(\ref{t:3}), it follows that 
\begin{equation} \label{e:n37b}
Q_{c} (\mb{P_T})\leq Q_{c} (\mb{\bar{P}})
\end{equation}
which implies that the independent variation of all parameters in LS optimization of the $Q_c$ variation is the most accurate  functional form to use assuming equal weightage of experimental measurements than the standard  free variation of parameters using  the $Q$ function of (\ref{e:6n}).
\section{Theory of method (i)b} \label{s:m(i)b}
Whilst it is advantageous in science  data-analysis  to  optimize a particular multiparameter function by focusing on a few key variables (our $k$ variable of  restricted dimensionality, which we have applied  to a 1-dimensional optimization in the next section) , it has been shown that this method yields a solution that is always of higher value  for the same $Q$ function  than a full  and independent parameter optimization, meaning that it is less accurate. The key issue, therefore, is whether for any $Q$  function, including those of the $Q_c$ variety,  it is possible to construct a $k$ parameter optimization such that the line  of parameter variables $\mb{P}(k)$ passes through the minimum surface of the $Q$ function. We develop a theory to construct such a function below. However, Method(i) may still be advantageous because of the greater simplicity of the equations to be solved, and the fact that $\mathcal{C}^1$ $f(i)$ functions were required, whereas here there  the $f(i)$ functions must be at least $\mathcal{C}^2$ continuous. 
\begin{theorem} \label{th2}
For the $Q_T(\mb{P},k)$ function defined in (\ref{e:6n}), where each of the $f(i)$ functions are $\mathcal{C}^2$  on an open set $\R^{N_p+1}$ and where $Q_T$ is convex, the solution at any point $k$ of $\frac{\partial Q_T}{\partial P_j}=0, (j=1,2,\ldots N_p) $ whenever $det \left \|\frac{\partial Q_T}{\partial P_i \partial P_j}  \right \|\neq 0$ at $\frac{dQ_T}{dk}(k^\prime)=0 $ determines uniquely the line equation $\mb{P}(k)$ that passes the minimum of the  function $Q_T$ when $k=k^\prime$.
\end{theorem}
\begin{proof}
As before $f(i)=f(\mb{P},t_i,k)$ so that 
\begin{equation} \label{e:n38}
Q_T=Q(\mb{P},k)=\sum_{i=1}^{N^\prime} (Y_{exp(t_i)}-f(i))^2.
\end{equation}
Define $\frac{\partial f(i)}{\partial P_j} =f(i,j)$, $\alpha_k=Y_{exp}(t_k)-f(k)$ and for an independent variation of the variables $(\mb{P},k)$ at the stationary point, we have 
\begin{eqnarray}
 \frac{\partial Q_T}{\partial P_j}&=& h_j(\mb{P},k)=c.\sum_{i=1}^{N^\prime}(Y_{exp}(t_i)-f(i))f(i,j)=0(j=1,2\ldots N_p)  \label{e:n39}\\
           \frac{\partial Q_T}{\partial k}  &=& I(\mb{P},k)=c.\sum_{i=1}^{N^\prime}\alpha_i \frac{\partial f(i)}{\partial k} =0    \label{e:n40}  
\end{eqnarray}
The above results for  the functions $h_j(\mb{P},k)=0\,(j=1,2\ldots N_p)$  to have a unique implicit function of $k$ denoted $\mb{P}(k)$  by the IFT \cite[Th.13.7,p.374]{apos1} requires that $det \left \|\frac{\partial h_i(\mb{P},k)}{\partial P_j}  \right \|=det \left \|\frac{\partial Q_T}{\partial P_i \partial P_j}  \right \|\neq 0$ on an open set $S,k\in S$. More formally, the expansion of the preceding determinant in (\ref{e:n41})   verifies that  a symmetric matrix  obtains for  $ \frac{\partial h_i(\mb{P},k)}{\partial P_j}$ due to the commutation of  second order partial derivatives of $P_j$
\begin{equation} \label{e:n41}
\frac{\partial h_i(\mb{P},k)}{\partial P_j}=c.\sum_l^{N^\prime}\left(\alpha_k\frac{\partial^2f(l)}{\partial P_j\partial P_i} -\frac{\partial f(l)}{\partial P_j}.\frac{\partial f(l)}{\partial P_i} \right )
\end{equation}
Defining $Q_1(k)$ as a function of $k$ only by expanding $Q_T$ yields the total derivative $w.r.t.\,k$  as      $Q_1^\prime(k)$  where \begin{eqnarray}
           Q_1(k)    &=& \sum_i^{N^\prime} (\alpha_i)^2 \label{e:n42}\\
            Q^\prime_1(k) &=& c.\sum_{j=1}^{N_p}\frac{dP_j}{dk}.\sum_{i=1}^{N^\prime}\alpha_i\frac{\partial f(i)}{\partial P_j}  +c.\sum_{i=1}^{N^\prime} \alpha_i \frac{\partial f(i)}{\partial k}  \label{e:n43}  
\end{eqnarray}
Then $h_i(\mb{P},k)=0$ by construction (\ref{e:n39}) so that $\frac{\partial Q_T}{\partial P_j}=0\,\,(j=1,2,\ldots N_p) $  and (\ref{e:n39}) implies $\sum_{i=1}^{N^\prime}\alpha_i\frac{\partial f(i)}{\partial P_j}=0\,\,(\forall j)$ and hence 
\begin{equation} \label{e:n44}
\left( \frac{dP_j}{dk}.\sum_{i=1}^{N^\prime}\alpha_i  \frac{\partial f(i)}{\partial P_j}   \right)=0.
\end{equation}
Substituting (\ref{e:n44}) derived from  (\ref{e:n39}) and (\ref{e:n40}) into (\ref{e:n43})  together with the condition $Q_1^\prime(k)=0$  implies that $c.\sum_{i=1}^{N^\prime}\alpha_i \frac{\partial f(i)}{\partial k}=0$, which satisfies (\ref{e:n40}) for the free variation in $k$. Thus $Q^{\prime}_1(k)=0 \Rightarrow \delta Q_T=0 $ for independent variation of $(\mb{P},k)$. So $Q_T$ fulfills the criteria of a stationary point at say $k=k_0$, since $\nabla_{\mb{P},k}Q_T=0$ (\cite[Prop. 16, p.112]{depree1}). Suppose that $Q_T$ is convex, where $P_T=\{\mb{P}_0,k_0\}$ is a minimum  point, $P_T\in D$, a convex subdomain of $Q_T$. Then at $P_T$, $\nabla_{\mb{P},k}Q_T=0$ , and $P_T$ is also  the unique global minimum over $D$ according to \cite[Theorem 3.2,pg.46]{crav1}. Thus $P_T$ is unique, whether derived from a free variation of $(\mb{P},k)$ or via $P(k)$ dependent parameters with the $Q_1$ function. 
\end{proof}
$Nb.$ As before, $Q_T$ and $Q_1$ may be replaced with the summation of indexes  as for $Q_c$ in (\ref{e:N30}) to derive a physically more accurate fit.
\section{Method (ii) theory} \label{sec:2c}
Methods (i)a and  (i)b, which are  mutual variants of each other  are  applications of the implicit method to modeling problems. Here, another variant of the implicit  methodology for optimization of   a target or cost function $Q_{E}$ is presented. One can for instance consider $Q_E(\mb{P},k)$ to be an energy function with  coordinates $\mb{R}=\{\mb{P},k\}$, where  as before the components of $\mb{P}$ is $P_j, j=1,2,...\ldots N_p$ , $k \in \R$ is another coordinate  so that  $\mb{R}\in \R^{N_p+1}$.  For bounded systems, (such as  the  molecular coordinates) , one can write
\begin{equation} \label{e:n44b}
l_{min,i}\leq P_i \leq l_{max,i}\,, i=1,2,...\ldots N_p ;\,\, l_{min,k}\leq k \leq l_{max,k}\,.
\end{equation}
Thus $\mb{R}\in \mc{D}\subset \R^{N_p+1}$ is in a compact space $\mc{D}$.
Define 
\begin{eqnarray}
 \frac{\partial Q_E}{\partial P_j}           &=& o_j(\mb{P},k)\,j=1,2....\ldots N_p  \label{e:n45}\\
          \frac{\partial Q_E}{\partial k}    &=& \kappa(\mb{P},k)   \label{e:n46}  
\end{eqnarray}
Then the equilibrium conditions become 
\begin{eqnarray}
            o_j(\mb{P},k)   &=&  0 \label{e:n47}\\
           \kappa (\mb{P},k)    &=& 0   \label{e:n48}  
\end{eqnarray}
Take (\ref{e:n45}) as the defining equations for $o_j(\mb{P},k)$ which is specified by $Q_E$ in (\ref{e:n45}) which casts it in a form compatible with the IFT where some further qualification is required for $(\mb{P},k)$. Assume $Q_E$ is $\mc{C}^2$ on $\mc{D}$, and $det\left\|  \frac{\partial o_i}{\partial P_j}   \right\|\neq 0$, where $\frac{\partial o_i}{\partial P_j} \equiv \frac{\partial^2 Q_E}{\partial P_j \partial P_i}$. The matrix of the aforementioned determinant is symmetric, partaking of the properties due to this fact. Then by the IFT \cite[TH. 13.7,p.374]{apos1}, $\exists $ a unique  $\mb{P}(k)$ function  where for some $k_0$, $o_j(\mb{P}(k_0),k_0)=0$ ,  $o_j \in \mc{C}^1$ on $\mc{T}_0$ with $(\mc{T}_0\times \R)\subseteq \mc{D}$  and $k_0\in \mc{T}_0$ such that $o_j(\mb{P}(k),k)=0$ for all $k\in \mc{T}_0$. For $k$ an isolated point $k=k_o$, from analysis,  we find $k$ is still open.  Write $Q_{E,1}(k)=Q_E(\mb{P}(k),k)$, and $Q^\prime_{E,1}(k)=\frac{dQ_{E,1}}{dk}$ , so that 
\begin{equation} \label{e:n49}
\frac{dQ_{E,1}}{dk} = \sum_{j=1}^{N_p}\left( \frac{\partial Q_E}{\partial P_j }.\frac{dP_j}{dk}   + \frac{\partial Q_E}{\partial k }   \right)= \sum_{j=1}^{N_p}\left( o_j.\frac{dP_j}{dk}   + \frac{\partial Q_E}{\partial k }   \right)
\end{equation}
Then denote $k_i$ as a solution to $Q^\prime_{E,1}(k)=0$ in the indicated range   where $k_i\in \mc{T}_0$  in the indicated range above in (\ref{e:n44b}).

\begin{theorem}\label{th3}
The stationary points $k_i (i=1,2,\dots N_k)$ where $Q_{E,1}(k)=0$  for $k=k_i$ exists for the range $\{k_i : k_{min} \leq  k \leq k_{max}\}$ of coordinate $k$ if and only if   for each  of these $k_i$,  (i) $Q_E(\mb{P},k)\in \mc{C}^2$ and (ii) $det \left\|  \frac{\partial Q_E(\mb{P},k_i)}{\partial P_j \partial P_i}    \right\|\neq 0 (\forall k_i, i=1,2,\ldots k_{max})$. Each of these points $k_i \in \R^1$ space corresponds uniquely in a local sense in the open set $\mc{T}_0$  to some equilibrium (stationary) point  of the target function $Q_E(\mb{P},k)$  in $\R^{N_p+1}$ space.
\end{theorem}
\begin{proof} 
If $Q^\prime_{E,1}(k)=0$    where $k\in \mc{T}_0$, then it also follows from the IFT that $o_j=0$, and therefore $\frac{\partial Q_E}{\partial k} = 0$ from (\ref{e:n49}), which satisfies (\ref{e:n45} )  and (\ref{e:n46}) for the equilibrium point. The conditions (i) and (ii) of the theorem is a requirement of the IFT.  Conversely, if $o_j=0, (j=1,2,\ldots N_p)$ and $\frac{\partial Q_E}{\partial k} =0$ (a stationary or equilibrium point), then by (\ref{e:n49}) $Q^\prime_{E,1}(k)=0$. Hence the coordinates $\{k_i\}$ for which $Q^\prime _{E,1}(k_i)=0$  refers to the condition where $\delta Q_E(\mb{P},k_i)=0$, and uniqueness follows from the IFT reference to the local uniqueness of the $\mb{P}(k)$ function.
\end{proof}
$Nb.$ In a bounded system, one can choose any of the $N_p$ components $P_j$ of $\mb{P}$ as the $k$ coordinate, partly based on the convenience of solving the implicit equations and determine the $k_i$ minima and thus determine by the uniqueness criterion the  coordinates of the minima  in $\R^{N_p+1}$ space. For non-degenerate coordinate choice, meaning that for a particular $k$  coordinate choice, there does not exist an equilibrium structure (meaning a set of coordinate values) where  for any two structures $A$ and $B$, $k_A=k_B$. For such structures, the total number of minima that exists within the bounded range in the $k$ coordinate equals the total number of minima of the target function $Q_E(\mb{P},k)$ within the bounded range. Hence a  method exists for the very challenging  problem of locating and enumerating minima \cite[Sec 5.1,p.242 "How many stationary points are there" ]{wales1} . From the uniqueness theorem of IFT, one  could infer points in the $k$  axis where non-uniqueness, obtains, i.e. whenever $det\left\|  \frac{\partial^2 Q_E}{\partial P_i \partial P_j}   \right\|=0$. In such cases, for particles with the same intermolecular potentials, permutation of the coordinates in conjunction with symmetry considerations could be of use in selecting the appropriate coordinate system to overcome these systems with degeneracies \cite[Sec. 4.2.5,p.205, "Appearance and dissappearance of symmetry elements"]{wales1}. Other methods  include scanning through one different $1-D$ graph  for  the $P_l$ coordinate to locate minima if relative to the $P_i$ coordinate, there exists for the same particular  $k_i$ value in the $P_i$ coordinate, there exists two structures with two different values for the $P_l$ coordinate.  Thus by  scanning through all or a select number of  the (1-D) $P_j$ profile for $Q_{E,1}$ , it would be possible to make an assign  of the location of a minima in $\R^{Np+1}$ space. One is reminded of the methods that spectroscopists use in assigning different energy  bands  based on selection rules to uniquely characterize for instance vibrational frequencies. A similar analogy obtains for X-ray reflections, where the amplitude variation of the X-ray intensity in  reciprocal space can be used to elucidate structure. The minima of the $1-D$ k coordinate scan must correspond to the minima in $\R^{Np+1}$ space of the $Q_E$ function given that all such minima in $Q_E$ are locally strict and global within  a small open set about the minima. By continuity, $Q_{E,1}(k)-Q_{E,1}(k_0) >  0$   for $\mid k-k_o\mid < \delta$ and for $\mid P(k)-P(k_0)\mid  <\delta_2$ which violate the condition for a maximum.
 
\section{Applications in Chemical Kinetics} \label{sec:3}
The utility of the  above triad of methods   is illustrated in  the determination of two parameters in chemical reaction rate studies, of $1^{\mbox{st} }$ and $2^{\mbox{nd} }$ order respectively using data from published literature , where method(i)a yields values close  within experimental error to those quoted in the literature.  The method  can directly derive certain parameters like the final concentration terms (e.g. $\lambda_\infty$ and $Y_\infty$  if $k$, the rate constant is the single optimizing variable in this approximation. We assume here that the rate laws and rate constants  are not slowly varying functions of  the reactant or product concentrations, which has recently from simulation been shown  to be generally not  the case \cite{cgj1}. Under this  standard assumption, the rate equations below all obtain.
The first order reaction studied here is 
(i) the methanolysis of ionized phenyl salicylate with data derived from the literature \cite[Table 7.1,p.381]{khan2} 
and the second order reaction analyzed is 
(ii)  the reaction between plutonium(VI) and iron(II) according to the data in \cite[Table II p.1427]{newt1} and \cite[Table 2-4, p.25]{bkreac15}.\\

\subsection{First order results}\label{subsec:2a}
Reaction (i) above  corresponds to 
 \begin{equation} \label{eq:1a}
\mbox{PS}^{-}+\mbox{CH$_3$OH}\,\, \stackrel{k_a}{\longrightarrow}\,\,  \mbox{MS}^{-}  + \mbox{PhOH} 
\end{equation}
 where  the rate law is pseudo first-order  expressed as  
\[\mbox{rate}=k_a\mbox{[PS]}^{-}=k_c[\mbox{CH$_3$OH}][\mbox{PS}^{-}]\]  
with   the concentration of methanol held constant  (80\% v/v) and where the physical and thermodynamical conditions   of the reaction appears in \cite[Table 7.1,p.381]{khan2}.
The change in  time $t$ for any material property  $\lambda(t)$, which in this case  is the Absorbance $A(t)$ (i.e. $A(t)\equiv \lambda(t)$) is given by 
\begin{equation} \label{eq:1}
\lambda(t)= \lambda _\infty -(\lambda_\infty - \lambda_0)\exp{(-k_at)}
\end{equation}
for a first order reaction  where $\lambda_0$ refers to the  measurable property value at time $t=0$ and $\lambda_\infty$ is the value at $t=\infty$ which is usually treated as a parameter to yield the best least squares fit  even if its optimized value is less for monotonically increasing functions (for positive $\frac{d\lambda}{dt}$at all $t$)  than an  experimentally determined $\lambda(t)$ at time $t$. In Table 7.1 of \cite{khan2} for instance, $A(t=2160 s)=0.897 >A_{opt,\infty}=0.882$  and this value of $A_\infty$ is used to derive the best  estimate of the rate constant as $16.5\pm 0.1\times 10^{-3}\mbox{sec}^{-1}$. \\
For this reaction, the $P_i$ of (\ref{eq:n2}) refers to $\lambda_\infty$  so that  $\mathbf{P}\equiv \lambda_\infty$  with $N_p=1$ and $k\equiv k_a$. To determine the parameter $\lambda_\infty$ as a function of $k_a$ according to (\ref{eq:n5}) based on the \emph{entire}  experimental $\{(\lambda_{exp},t_i)\}$ data set we invert (\ref{eq:1}) 
and write 
\begin{equation} \label{eq:1c}
\lambda_{\infty}(k)=\frac{1}{N^\prime}\sum_{i=1^\prime}^{N^\prime}\frac{(\lambda_{exp}(t_i)- \lambda_{o}\exp{-kt_i} )}{(1-\exp{-kt_i})}
\end{equation}
where the summation is for all the values of the experimental dataset that does not lead to singularities, such as when $t_i=0$, so that here $N_s=1$. We define the non-optimized, continuously deformable theoretical curve $\lambda_{th}$ where $\lambda_{th}\equiv Y_{th}(t,k)$ in (\ref{eq:n3}) as 
\begin{equation} \label{eq:1d}
\lambda_{th}(t,k)= \lambda _\infty(k) -(\lambda_\infty(k) - \lambda_0)\exp{(-k_at)}
\end{equation}
With such a relationship of the  $\lambda_{\infty}$ parameter $P$ to $k$, we seek the  least square minimum of $Q_1(k)$, where $Q_1(k)\equiv Q$ of (\ref{eq:n5}) for this first-order rate constant k  in the form  
\begin{equation} \label{eq:1e}
Q_1(k)=\sum_{i=1}^N (\lambda_{exp}(t_i) -\lambda_{th}(t_i,k))^2
\end{equation}
where the summation is over all the experimental $(\lambda_{exp}(t_i), t_i)$ values. 
The resulting $P_k$ function (\ref{eq:n6})  for the first order reaction based on  the published dataset is given in Fig.(\ref{fig:1n}). The solution of the rate constant $k$ corresponds to the zero value of the function, which exists for both orders. The $\mathbf{P}$ parameters ($\lambda_\infty$ and $Y_\infty$ ) are derived by back substitution into eqs. (\ref{eq:1c}) and (\ref{eq:3}) respectively. The Newton-Raphson (NR) numerical procedure \cite[p.456]{nrc} was used to find the roots to $P_k$. For each dataset, there exists a value for $\lambda_\infty$ and so the error expressed as a standard deviation may be computed. The tolerance in accuracy for the NR procedure was $1.\times 10^{-10}$ . We define the function deviation $fd$ as the standard deviation of the  experimental results with the best fit curve $fd=\surd\frac{1}{N} \{\sum_{i=1}^N(\lambda_{exp}(t_i)-\lambda_{th}(t_i)^2\}$ Our results are as follows:\\
 $k_a=1.62\pm.09\times10^{-2}\mbox{s}^{-1}$; $\lambda_\infty=0.88665\pm.006$; and $fd=3.697\times 10^{-3}$. \\ The experimental estimates are :\\
  $k_a=1.65\pm.01\times10^{-2}\mbox{s}^{-1}$; $\lambda_\infty=0.882\pm 0.0$; and  $fd=8.563\times 10^{-3}$.\\
 The experimental method involves adjusting the $A_\infty\equiv \lambda_\infty$ to minimize the  $fd$ function and hence no estimate of the error in $A_\infty$ could be made. It is clear that our method has a lower $fd$  value and is thus a better fit, and the parameter values can be considered to coincide with the experimental estimates  within experimental error. Fig.(\ref{fig:2n}) shows the close fit between the curve due to our optimization procedure and experiment. The slight variation between the two curves may well be due to experimental uncertainties.  
\begin{figure}[htbp]
\begin{center}
\includegraphics[width=9cm]{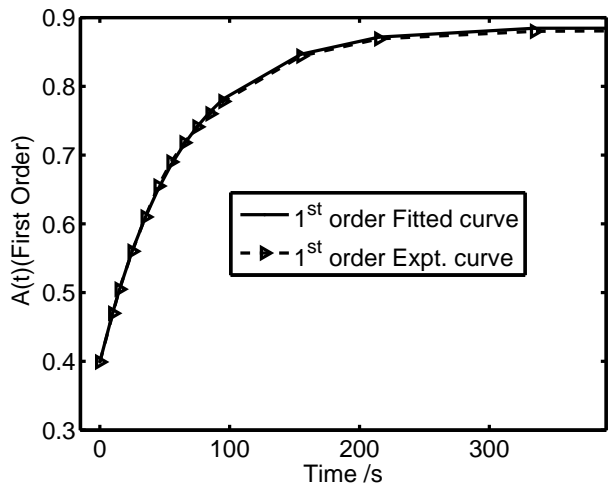} 
\end{center}
\caption{Plot of the experimental and curve with optimized parameters showing the very close fit between the two. The slight difference between the two can  probably be   attributed to experimental errors. }
\label{fig:2n} 
\end{figure}
 
\begin{figure}[htbp]
\begin{center}
\includegraphics[width=7cm]{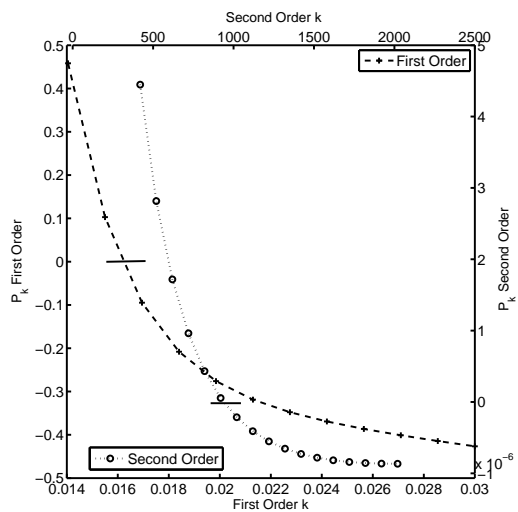} 
\end{center}
\caption{$P_k$ functions  (\ref{eq:n6})  for reactions (i) and (ii) of order one and two in reaction rate.  }
\label{fig:1n} 
\end{figure}
\subsection{Second order results}\label{subsec:2b}
To further test our method, we also analyze the second order reaction
 \begin{equation} \label{eq:1b}
\mbox{Pu(VI)}+2\mbox{Fe(II)}\,\, \stackrel{k_b}{\longrightarrow}\,\,  \mbox{Pu(IV)}  + 2\mbox{Fe(III)} 
\end{equation}
whose rate is given by $\mbox{rate}=k_0[\mbox{PuO}^{2+}_2][\mbox{Fe}^{2+}]$ where $k_0$ is relative to the constancy of other ions in solution such as $\mbox{H}^+$. The equations are very different in form to the first-order expressions and serves to confirm the viability of the current method.

For  Espenson, the above stoichiometry  is kinetically equivalent to the reaction scheme  \cite[eqn. (2-36)]{bkreac15}
\[\mbox{PuO}_{2}^{2+}+ \mbox{Fe}^{2+}_{aq}   \stackrel{k_b}{\longrightarrow}\, \mbox{PuO}^+_2 +\mbox{Fe}^{3+}_{aq}.\]  
which also follows from the work of  Newton et al. \cite[eqns. (8,9),p.1429]{newt1}  whose data \cite[TABLE II,p.1427]{newt1} we use and analyze to  verify  the principles presented here. Espenson had also used the same data as we have to derive the rate constant and other parameters \cite[pp.25-26]{bkreac15} which is used to  check the accuracy of our methodology. The overall absorbance in this case $Y(t)$ is given by \cite[eqn(2-35)]{bkreac15}
\begin{equation} \label{eq:2}
Y(t)= \frac{ Y_{\infty} + \left\{ Y_0 \left( 1-\alpha \right) -Y_\infty    \right\}\exp(-k\Delta_0 t)   }
{1-\alpha\exp(-k\Delta_0t)}
\end{equation}
where $\alpha=\frac{[\text{A}]_0}{[\text{B}]_0}$ is the  ratio  of initial concentrations where $[\text{B}]_0>[\text{A}]_0$ and $[\text{B}]=[\mbox{Pu(VI)}]$, $[\text{A}]=[\mbox{Fe(II)}]$ and $[\text{B}]_0=4.47\times 10^{-5}\text{M}$ and $[\text{A}]_0=3.82\times 10^{-5}\text{M}$ . A rearrangement of (\ref{eq:2}) leads to  the equivalent expression \cite[eqn(2-34)]{bkreac15}

\begin{equation} \label{eq:3}
\ln \left\{ 1+ \frac{\Delta_0\left(Y_0-Y_\infty\right)}{[\text{A}]_0\left(Y_t - Y_\infty \right)}\right\}
=\ln\frac{[\text{B}]_0}{[\text{A}]_0} +k\Delta_0t. 
\end{equation}
According to Espenson, one cannot use this equivalent form \cite[p.25]{bkreac15} "because an experimental value of $Y_\infty$ was not reported."  However, according to Espenson, if $Y_\infty$ is  determined autonomously, then $k$ the rate constant may be determined. Thus, central to all conventional methods is the autonomous and independent status of both $k$ and $Y_\infty$.  We overcome this interpretation by defining $Y_\infty$ as a function of the total experimental spectrum of $t_i$ values and $k$  by inverting (\ref{eq:2}) to define $Y_{\infty}(k)$ where 
\begin{equation} \label{eq:3b}
Y_{\infty}(k)=\frac{1}{N^\prime}\sum_{i=1^\prime}^{N^\prime}\frac{Y_{exp}(t_i)\left\{\exp(k\Delta_0t_i)-\alpha) \right\}+Y_0(\alpha-1)}{(\exp(k\Delta_0t_i)-1)}
\end{equation}
where the summation is over all experimental values that does not lead to singularities such as at $t_i=0$.
In this case, the $\mathbf{P}$ parameter is given by Y$_{\infty}(k)=P_1(k)$, $k_b=k$ is the varying $k$ parameter of (\ref{eq:n2}).  We likewise define a  function $Y_{th}$ of $k$  that is also a function of t, but where the $k$ parameter is interpreted as  a "distortion"  parameter in the following manner:
\begin{equation} \label{eq:3c}
Y(t,k)_{th}= \frac{ Y_{\infty}(k) + \left\{ Y_0 \left( 1-\alpha \right) -Y_\infty(k)    \right\}\exp(-k\Delta_0 t)   }
{1-\alpha\exp(-k\Delta_0t)}.
\end{equation}
 In order to extract the parameters $k$ and $Y_{\infty}$ we minimize
 the  square function   $Q_2(k)$  for this second order  rate constant with respect to $k$  given as 
\begin{equation} \label{eq:3d}
Q_2(k)=\sum_{i=1}^N (Y_{exp}(t_i) -Y_{th}(t_i,k))^2
\end{equation}
where the summation is over the experiment $t_i$ coordinates. Then the solution to the minimization problem is when the corresponding $P_k$ function (\ref{eq:n6}) is zero. The NR method was  used to  solve $P_k=0$ with the error tolerance of $1.0\times10^{-10}$.
With the same notation as in the first order case, the second order results are:\\
  $k_b=938.0\pm 18 \mbox{(M s)}^{-1}$; $Y_\infty=0.0245 \pm 0.003$; and $fd=9.606\times 10^{-4}$. \\
 
The experimental estimates are \cite[p.25]{bkreac15}:\\
  $k_b=949.0\pm 22\mbox{(M s)}^{-1}$; $Y_\infty=0.025 \pm 0.003$.\\

  Again the two results are in close agreement. The graph of the experimental curve and the one that derives from our optimization method in given in Fig.(\ref{fig:3n}).
  
\begin{figure}[htbp]
\begin{center}
\includegraphics[width=9cm]{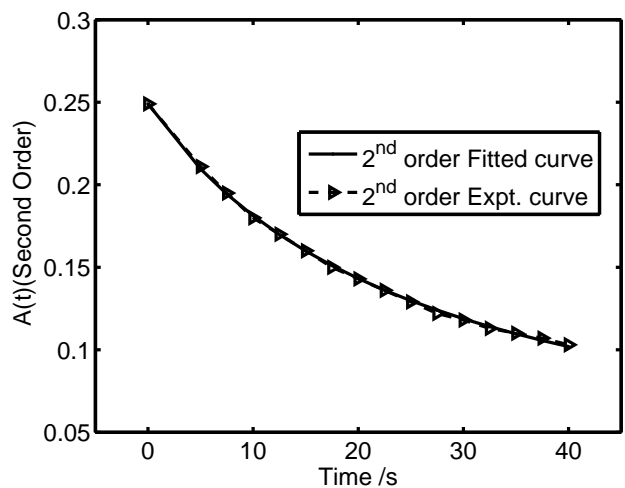} 
\end{center}
\caption{Graph of the experimental and calculated curve based on the current induced parameter-dependent optimization method. }
\label{fig:3n} 
\end{figure}

\section{Conclusions}
The triad of associated implicit function optimization covers both modeling of data \textit{and} the optimization of arbitrary functions where experimental or theoretical considerations require  that a single variable is tagged  to a process variable that is iteratively relaxing to equilibrium. Applying method (i)a to chemical kinetics allows for the direct determination of parameters not possible by the standard methodologies used.
The results presented here show that for linked variables, it is possible to derive all the parameters associated with a curve by considering only one independent variable which serves as the independent variable for  other functions in the optimization process that uses the experimental dataset as function values   in the estimation. Apart from possible reduced errors in  the computations, it might also be  a more accurate way of deriving parameters that are more influenced or conditioned (on physical grounds)  by the value of one parameter (such as $k$ here) than others; the current methods that gives equal weight  to all the variables might in some cases lead to results that would be considered "unphysical". In complex dynamical systems with multiprocesses, the physical considerations are such that for scientific purposes, it would be advantageous if optimization would be conducted on just one primary coordinate variable, such as in attempting to derive the most general stable conformer in a large molecule, where there are thousands of local minima present if all free coordinate variables are considered \cite[Sec.6.7, p.330]{wales1} .  This generalized potential surface might be found   suitable for  reaction trajectory calculations \cite[Ch.4, p.192 on "Features of a landscape"]{wales1} that require a single path variable, where the  general optimized conformer  would be relevant to the study of the potential surfaces and force fields present.

\section{Acknowledgments}
This work was supported by University of Malaya Grant UMRG(RG077/09AFR) and Malaysian Government grant   FRGS(FP084/2010A). Cordial discussions concerning real world applications with Gareth Tribello (ASC) is acknowledged. I thank Jorge Kohanoff (ASC) and Christopher Hardacre (Chem. Dept., QUB) for congenial hospitality during this sabbatical. \\ \\

\bibliographystyle{unsrt}	% or "unsrt", "alpha", "abbrv", etc.
%%\bibliography{mpbib}
\bibliography{mpbib}
\end{document}